\documentclass[12pt,leqno]{article}
\usepackage[dvips]{graphicx}
\usepackage{amsmath,amssymb,amsthm,amscd}
\usepackage[mathscr]{eucal}
\usepackage{amsfonts}
\begin{document}
\parskip=6pt
\newtheorem{prop}{Proposition}
\numberwithin{prop}{section}
\newtheorem{thm}{Theorem}
\numberwithin{thm}{section}
\newtheorem{corr}{Corollary}
\numberwithin{corr}{section}
\newtheorem{lemma}{Lemma}
\newtheorem{defn}{Definition}
\numberwithin{defn}{section}
\numberwithin{lemma}{section}
\numberwithin{equation}{section}
\newtheorem{proof2}{Proof of Lemma 2.2}
\newtheorem{proof3}{Proof of Theorem 3.1}
\newtheorem{proof5}{Proof of Lemma 5.2}
\newtheorem{proof4}{Proof of Theorem 1.4}
\newcommand{\cF}{\cal F}
\newcommand{\cA}{\cal A}
\newcommand{\cC}{\cal C}
\newcommand{\cO}{\cal O}
\newcommand{\var}{\varepsilon}
\newcommand{\bC}{\mathbb C}
\newcommand{\bP}{\mathbb P}
\newcommand{\bN}{\mathbb N}
\newcommand{\bA}{\mathbb A}
\newcommand{\bR}{\mathbb R}
\newcommand{\fU}{\frak U}
\newcommand{\hol}{\text{hol}}
\newcommand{\Hom}{\text{Hom}}
\newcommand{\id}{\text{id}}
\newcommand{\Ker}{\text{Ker}\,}
\newcommand{\im}{\text{Im}\,}
\renewcommand\qed{ }
\begin{titlepage}
\title{\bf Representing Analytic Cohomology Groups of Complex Manifolds\thanks{Research supported in part by the NSF grant DMS-1162070\newline
2010 Mathematics Subject Classification 32E10, 32L10, 46G20}}
\author{L\'aszl\'o Lempert\\ Department of  Mathematics\\
Purdue University\\West Lafayette, IN
47907-2067, USA}
\end{titlepage}
\date{}
\maketitle
\abstract
Consider a holomorphic vector bundle $L\to X$ and an open cover $\fU=\{U_a\colon a\in A\}$ of $X$, parametrized by a complex manifold $A$.
We prove that the sheaf cohomology groups $H^q(X,L)$ can be computed from the complex $C^{\bullet}_{\hol}$ $(\fU,L)$ of cochains 
$(f_{a_0\ldots a_q})_{a_0,\ldots, a_q\in A}$ that depend holomorphically on the $a_j$, provided $S=\{(a,x)\in A\times X\colon x\in U_a\}$ is a Stein open subset of $A\times X$.
The result is proved in the setting of Banach manifolds, and is applied to study representations on cohomology groups induced by a holomorphic action of a complex reductive Lie group on $L$.
\endabstract

\section{Introduction}

Consider a holomorphic vector bundle $L\to X$.
Its cohomology groups $H^q(X,L)$ are often represented in terms of open covers $\fU=\{U_a\colon a\in A\}$ of $X$ and the associated \v Cech complex $C^{\bullet}(\fU,L)$, whose elements are collections $(f_{a_0\ldots a_q})_{a_j\in A}$, with each $f_{a_0\ldots a_q}\in\Gamma (\bigcap^q_{j=1} U_{a_j},L)$ a holomorphic section of $L$.
If each $U_a$ is Stein, by Cartan's Theorem B and by Leray's theorem $H^q(X,L)\approx H^q(C^{\bullet}(\fU,L))$, see \cite{Ca,B}.

The notion of \v Cech cochains $(f_{a_0\ldots a_q})$ is very natural if the cover $\fU$ is indexed by a set $A$ without any structure.
However, as noted in \cite{G,BEG1,BEG2,BE,EGW}, if $A$ has some structure, then it makes sense to consider cochains that, in their dependence on $a_j$, reflect this structure.
For example, if $A$ is a differential or complex manifold, or a measure space, one can work with the subspaces $C^{\bullet}_{\text{smooth}}(\fU,L)$, $C^{\bullet}_{\hol}(\fU,L)$ or $C^{\bullet}_{\text{meas}}(\fU,L)$ of cochains $(f_{a_0\ldots a_q})$ that depend smoothly, holomorphically, or measurably on $a_0,\ldots,a_q$.
In this paper we prove that under a certain condition the holomorphic \v Cech complex $C^{\bullet}_{\hol}(\fU,L)$ and $C^{\bullet}(\fU,L)$ have isomorphic cohomology groups.

\begin{thm}Let $A,X$ be complex manifolds, $L\to X$ a holomorphic vector bundle, and $\fU=\{U_a\colon a\in A\}$ an open cover of $X$.
If
$$
S=\{(a,x)\in A\times X\colon x\in U_a\}\subset A\times X
$$
is a Stein open subset, then inclusion $C^{\bullet}_{\hol}(\fU,L)\subset C^{\bullet}(\fU,L)$ induces an isomorphism of cohomology groups.
\end{thm}

Covers parametrized by complex manifolds occur in many situations. A natural Stein cover of projective space $\bP$ is by complements of hyperplanes. This cover is parametrized by the hyperplanes, i.e., by points of the dual projective space $\bP^*$. By restriction, we also obtain a Stein cover of any projective manifold $X\subset\bP$, parametrized by $\bP^*$. These covers satisfy the assumptions of Theorem 1.1.

The theorem is related to [EGW, Theorem 1.1], see also \cite{G}.
There it is assumed additionally that the sets $\{a\in A\colon x\in U_a\}$ are contractible, and the conclusion is that $H^q(\fU,L)$, or $H^q(X,L)$, is isomorphic to a certain relative holomorphic De Rham cohomology group.
As in \cite{EGW,G}, the motivation to represent cohomology groups through holomorphic objects comes from the study of bundles on which a complex reductive group acts holomorphically, see Theorem 5.5.

So far we have been vague about the sort of complex manifolds and vector bundles covered by Theorem 1.1.
In fact, while the theorem is new even for finite dimensional $L$, it  holds for a large class of Banach manifolds $A,X$ and Banach bundles $L\to X$; and the isomorphism in the theorem is that of topological vector spaces.
In Section 2 we will explain the necessary background in infinite dimensional complex geometry and in Sections 3 and 4 we formulate and prove the infinite dimensional version of Theorem 1.1 (Theorem 3.1).
Establishing a very special case of this theorem was the first step in \cite{LZ} of the computation of the first cohomology group of various loop spaces $L\bP_1$ of the Riemann sphere (in guise of the Dolbeault group $H^{0,1}(L\bP_1)$).
Finally, in Section 5 we apply Theorem 3.1 to the study of holomorphic group actions. We hope these results will pave the way to the computation of higher cohomology groups of loop spaces of projective spaces $\bP_n$ and more generally, of projective manifolds.

\section{Complex Banach manifolds}

In this section we recall basic notions of infinite dimensional complex geometry as well as key vanishing and isomorphism theorems.
The main references are \cite{D,L,LP,M}.

Let $E,F$ be Hausdorff, locally convex topological vector spaces over $\bC$, $F$ sequentially complete, and $\Omega\subset E$ open.
A function $f\colon\Omega\to F$ is holomorphic if at every $x\in\Omega$ the directional derivatives
$$
df(x;v)=\lim_{\bC\ni t\to 0}\ {f(x+tv)-f(x)\over t }
$$
exist, and define a continuous map $df\colon\Omega\times E=T\Omega\to F$.

A complex manifold in this paper will be a Hausdorff space sewn together from open subsets of Banach spaces with holomorphic sewing maps.
A closed subset $Y$ of a complex manifold $X$ is a direct submanifold if for every $x\in Y$ there are neighborhoods $U\subset X$, a Banach space $E$, a complemented subspace $F\subset E$, and a neighborhood $V\subset E$ of $0\in E$ such that the pair $(U,U\cap Y)$ is biholomorphic to $(V,V\cap F)$.
Such a $Y$ has a natural structure of a complex manifold.
A holomorphic Banach bundle is a holomorphic map $\pi\colon L\to X$ of complex manifolds, each fiber $L_x=\pi^{-1}(x)$ endowed with the structure of a complex vector space.
It is required that for each $x\in X$ there be a neighborhood $U\subset X$, a Banach space $E$, and a biholomorphism $L_U=\pi^{-1}U\to U\times E$ that for $y\in U$ maps the fiber $L_y$ linearly on $\{y\}\times E$.
We denote by $\Gamma (X,L)$ the vector space of holomorphic sections of $L$.

An open subset $\Omega$ of a Banach space $E$ is pseudoconvex if $\Omega\cap E'\subset E'$ is pseudoconvex for all finite dimensional subspaces $E'\subset E$.
A connected complex manifold is Stein if it is biholomorphic to a direct submanifold of a pseudoconvex subset $\Omega\subset E$, where $E$ is a Banach space with a Schauder basis.
A general complex manifold is Stein if its connected components are.

As we shall see, on Stein manifolds generalizations of Cartan's Theorems A and B hold.
More generally, we shall consider locally Stein manifolds $X$ in which every point has a Stein neighborhood.
This is equivalent to $X$ being modeled on complemented subspaces of Banach spaces with a Schauder basis.

The sheaf of germs of holomorphic sections of a holomorphic Banach bundle $L\to X$ is a {\sl cohesive} sheaf, (see [L, Definition 3.5]), whose theory was developed in \cite{LP,L}.
We now review this theory in the simpler context of bundles.

\begin{defn}A homomorphism $\varphi$ between holomorphic Banach bundles $L\to X$ and $L'\to X$ is a complete epimorphism if for every trivial bundle $T\to X$ and for every Stein open $U\subset X$ the induced map
$$
\varphi_*\colon\Gamma (U,\Hom(T,L))\to \Gamma(U,\Hom(T,L'))
$$
is surjective.
\end{defn}

\setcounter{lemma}{1}
\begin{lemma}If $\Lambda,L\to X$ are holomorphic Banach bundles over a Stein manifold $X$, then a homomorphism $\varphi\colon\Lambda\to L$ is a complete epimorphism if and only if there is a homomorphism $\psi\colon L\to\Lambda$ such that $\varphi\psi=\id_L$.
In this case $\psi(L)$, $\Ker \varphi\subset\Lambda$ are subbundles, and $\Lambda=\psi(L)\oplus\Ker \varphi$.
\end{lemma}

We shall derive the lemma from

\setcounter{thm}{2}
\begin{thm}If $L\to X$ is a holomorphic Banach bundle over a Stein manifold, then

(a)\ Some trivial Banach bundle has a complete epimorphism on $L$;

(b)\ $H^q(X,L)=0$ for $q\geq 1$.
\end{thm}

This is a special case of the sheaf theoretic [L, Theorem 3.7 and Lemma 3.8], which, in turn, depended on \cite{LP,P}.
Accepting Lemma 2.2, part (a) above means there is a holomorphic Banach bundle $L'\to X$ such that $L\oplus L'$ is trivial.

\begin{proof}[ Proof of Lemma 2.2]
The ``if'' part being obvious, we only prove the ``only if'' direction.
Suppose first that $\Lambda=X\times E\to X$ and $L=X\times F\to X$ are trivial.
We choose $\psi\in\Gamma (X,\Hom(L,\Lambda))$ to have image $\id_L\in\Gamma(X,\Hom(L,L))$ under the surjective map 
$$
\varphi_*\colon \Gamma (X,\Hom(L,\Lambda))\to\Gamma(X,\Hom(L,L)).
$$
This $\psi$ determines a homomorphism $L\to\Lambda$, also denoted $\psi$; clearly $\varphi\psi=\id_L$.
This implies
\begin{equation}
\Lambda_x=\psi(L_x)\oplus \Ker \varphi(x)\text{ for }x\in X.
\end{equation}

Next we show $\psi(L)$ and $\Ker\varphi$ are subbundles of $\Lambda$.
Fix $x_0\in X$, let $K=\Ker\varphi(x_0)\subset E$, write
$$
\psi(x,u)=(x,\Psi(x,u)),\qquad x\in X, u\in F,\Psi(x,u)\in E,
$$
and define a homomorphism of trivial bundles
$$
\vartheta\colon X\times (F\oplus K)\ni (x,u,v)\mapsto (x,\Psi(x,u)+v)\in X\times E.
$$
The inverse function theorem implies $\vartheta$ is an isomorphism of Banach bundles over a neighborhood of $x_0$, and it follows from (2.1) that
$$
\psi(L)=\vartheta(X\times (F\oplus 0))\text{ and }\Ker\varphi=\vartheta (X\times (0\oplus K))
$$
are complementary subbundles of $\Lambda$, near $x_0$.
Since $x_0$ was arbitrary, this holds over all of $X$.

To complete the proof, consider general Banach bundles $\Lambda,L$.
Since locally $\Lambda,L$ are trivial, $\Ker\varphi\subset\Lambda$ is a subbundle that is locally complemented, whence $\varphi$ gives rise to an exact sequence
$$
0\to\Hom (L,\Ker\varphi)\to \Hom (L,\Lambda)\to\Hom (L,L)\to 0
$$
of Banach bundles.
As $H^1(X,\Hom (L,\Ker\varphi))=0$ by Theorem 2.3, the associated long exact sequence gives that $\varphi_*\colon\Gamma (X,\Hom (L,\Lambda))\to\Gamma (X,\Hom (L,L))$ is surjective, and we can proceed as above when $L,\Lambda$ were trivial.
\end{proof}

Next we turn to defining a locally convex topology on the space of sections of a holomorphic Banach bundle $L\to X$ and on its cohomology groups, following \cite{L}.
By a weight on a complex manifold $X$ we mean a locally bounded function $w\colon X\to (0,\infty)$, and we denote by $W(X)$ the set of all weights.
This is a directed set with the partial order $w>w'$ meaning $w(x)>w'(x)$ for all $x\in X$.
If $(E,\|\ \|)$ is a Banach space, we write ${\cO}^E(X)$ for the space of holomorphic functions $X\to E$, and if $w\in W(X)$,
$$
{\cO}^E(w)=\{f\in{\cO}^E(X)\colon \|f\|_w=\sup_{x\in X} \|f(x)\|/w(x)<\infty\}.
$$
Thus $({\cO}^E(w),\|\ \|_w)$ is a Banach space.
The $\tau_\delta$ topology on ${\cO}^E(X)$ is the locally convex direct limit topology of the ${\cO}^E(w)$, see [L, Proposition 5.1].
Basic neighborhoods of $0\in{\cO}^E(X)$, parametrized by functions $\varepsilon\colon W(X)\to (0,\infty)$, are the convex hulls of sets of form $\bigcup_{w\in W(X)}\{f\in {\cO}^E(w)\colon \|f\|_w < \varepsilon (w)\}$.
From this description it is clear that $\tau_\delta$ is Hausdorff.
The $\tau_\delta$ topology was first introduced by Nachbin in \cite{N}; a variant, that often agrees with it, was studied by Coeur\'e in \cite{Co}.
When $X$ is finite dimensional and second countable, $\tau_\delta$ is the same as the compact--open topology, but in general $\tau_\delta$ is finer.

As sections of a trivial bundle $T=X\times E\to X$ are in one to one correspondence with functions $X\to E$, we obtain a topology on $\Gamma(X,T)$, also denoted $\tau_\delta$.
Suppose now $L\to X$ is a holomorphic Banach bundle over a Stein manifold.
By Theorem 2.3 and by Lemma 2.2 we can choose a trivial bundle $T\to X$ in which $L$ is a direct summand.
Projection $p\colon T\to L$ induces a surjection $p_*\colon\Gamma (X,T)\to\Gamma (X,L)$, and we define the $\tau_\delta$ topology on $\Gamma(X,L)$ as the finest topology for which $p_*$ is continuous.
As discussed in [L, Section 4], this topology is independent of the choice of $T$ and of the complete epimorphism $p$.
A homomorphism $\varphi\colon L\to L'$ induces a continuous map $\Gamma(X,L)\to\Gamma(X,L')$, see [L, Proposition 4.3].
Hence:

\setcounter{prop}{+3}
\begin{prop}The map
$$
\Gamma(X,L)\oplus\Gamma (X,L')\ni (f,f')\mapsto f\oplus f'\in \Gamma(X,L\oplus L')
$$
is a topological isomorphism.
In particular, $\Gamma(X,L)\ni f\mapsto f\oplus 0\in\Gamma (X,L\oplus L')$ is a topological embedding.
\end{prop}

We apply this when $L'$ is chosen so that $L\oplus L'=T$ is trivial.
Since $\Gamma(X,T)$ is Hausdorff, and sequentially complete by [L, Theorem 7.1], we obtain:

\begin{prop}The $\tau_\delta$ topology on $\Gamma(X,L)$ is Hausdorff and sequentially complete.
\end{prop}

\cite{L} also introduces a so called $\tau^\delta$ topology on $\Gamma(X,L)$ when $X$ is just locally Stein, but we will not need it here.
Next let $L\to X$ be a holomorphic Banach bundle over a locally Stein manifold and $\fU=\{U_a\colon a\in A\}$ a cover of $X$ by Stein open subsets.
As usual, if $q\geq 0$ and $a=(a_0,\ldots,a_q)\in A^{q+1}$ is a $q$--simplex, we write $U_a=U_{a_0}\cap\ldots\cap U_{a_q}$; by [L, Proposition 3.1] $U_a$ is Stein.
We also introduce a $(-1)$ simplex $a=\emptyset$, which constitutes $A^0$, and set $U_\emptyset=X$.
The disjoint union $\fU_q=\coprod_{a\in A^{q+1}}U_a$ is a Stein manifold when $q\geq 0$.
We denote by $\rho_q\colon\fU_q\to X$ the local biholomorphism for which $\rho_q|U_a$ is the embedding $U_a\hookrightarrow X$.
There is a natural vector space isomorphism between $\Gamma(\fU_q,\rho_q^* L)$ and the space
$$
C^q(\fU,L)=\prod_{a\in A^{q+1}}\Gamma (U_a,L)
$$
of not necessarily alternating cochains
$$
\Gamma(\fU_q,\rho_q^* L)\ni f\mapsto (f|U_a)_{a\in A^{q+1}}\in C^q (\fU,L),
$$
and we define the topology on $C^q(\fU,L)$ as the image of the topology on $\Gamma(\fU_q,\rho_q^* L)$.
\v Cech coboundary $\delta=\delta^q\colon C^q (\fU,L)\to C^{q+1}(\fU,L)$ is continuous, and the cohomology groups $H^q (\fU,L)=\Ker\delta^q/\im\delta^{q-1}$ are given the subquotient topology.
This is a locally convex topology but not necessarily Hausdorff.
The topology on \v Cech cohomology groups $\check H^q (X,L)=
\underset{\longrightarrow}{\lim} \ H^q(\fU,L)$ is the direct limit topology, the finest locally convex topology for which the canonical maps
\begin{equation}
H^q (\fU,L)\to\check H^q (X,L)
\end{equation}
are continuous.
According to the main theorem of \cite{L}, Theorem 4.5 there, (2.2) is in fact a topological isomorphism.
([L, Theorem 4.5] applies only to so--called separated cohesive sheaves, [L, Definition 4.1], but the sheaf of holomorphic sections of a Banach bundle is separated by [L, Lemma 4.2], as separation is a local property.)

\section{Holomorphic cochains}

Let $L\to X$ be a holomorphic Banach bundle, $\fU=\{U_a\colon a\in A\}$ an open cover of $X$ as before, but suppose now that $A$ itself is a complex manifold and that
$$
S=\{(a,x)\in A\times X\colon x\in U_a\}
$$
is a Stein open subset of $A\times X$.
It follows that $U_a=\{x\in X\colon (a,x)\in S\}$ are Stein, hence $X$ is locally Stein.
It also follows that for $p=0,1,\ldots$ the fiber product
\begin{equation}
S^{[p+1]}=\{(a_0,\ldots,a_p,x)\in A^{p+1}\times X\colon x\in U_{a_0}\cap \ldots\cap U_{a_q}\}
\end{equation}
is a Stein submanifold of $S^{p+1}$, and in fact a Stein open subset of $A^{p+1}\times X$.
We put $S^{[0]}=X$.
Let $\pi_p\colon S^{[p+1]}\to X$ denote the projection.
The space $C_{\hol}^p (\fU,L)\subset C^p (\fU,L)$ of cochains $(f_{a_0\ldots a_p})$ depending holomorphically on $a_0,\ldots,a_p$ can be identified with $\Gamma(S^{[p+1]},\pi_p^* L)$, where to $(f_{a_0\ldots a_p})$ corresponds $f$ defined by $f(a_0,\ldots,a_p,x)=f_{a_0\ldots a_p}(x)$.

There are two natural topologies on $C_{\hol}^p (\fU,L)$:\ the one inherited as a subspace of $C^p(\fU,L)$, and the one coming from identification with $\Gamma(S^{[p+1]},\pi_p^* L)$, the latter being the finer.
Accordingly, there are two ways to induce topology on $H^p_{\hol} (\fU,L)=H^p (C^\bullet_{\hol}(\fU,L))$, to which we refer as the coarser and finer topologies.
In fact, the two coincide:

\begin{thm}Inclusion $C^\bullet_{\hol}(\fU,L)\hookrightarrow C^\bullet (\fU,L)$ induces topological isomorphisms of cohomology groups $H_{\hol}^n (\fU,L)\to H^n (\fU,L),n=0,1,\ldots$, whether the former is endowed with the finer or the coarser topology.
The same holds if the cochain complexes are replaced by the complexes of alternating cochains.
\end{thm}

In other words, inclusion $C^\bullet_{\hol} (\fU,L)\to C^\bullet (\fU,L)$ is a topological quasi--isomorphism.

\section{The proof of Theorem 3.1}

Predictably, Theorem 3.1 will follow from the study of double complexes.
For $p,q\geq 0$ let
\begin{eqnarray}
K^{pq}=C^q (\pi_p^{-1} \fU,\pi_p^* L)\supset J^{pq}=C^q_{\hol}(\pi_p^{-1} \fU,\pi_p^* L)
\end{eqnarray}
be spaces of cochains on $S^{[p+1]}, K=(K_{pq})_{p,q\geq 0}$, $J=(J_{pq})_{p,q\geq 0}$.
There are differentials 
$$
\delta^{pq}\colon K^{pq}\to K^{p,q+1}\text{ and }\partial^{pq}\colon K^{pq}\to K^{p+1,q},
$$
the first \v Cech coboundary, the second fiberwise Alexander--Spanier coboundary.
That is, if $f=(f_a)\in K^{pq}$ then
\begin{gather}
(\delta f)_{a_0\ldots a_{q+1}}=\sum_0^{q+1} (-1)^i f_{a_0\ldots \hat a_i\ldots a_{q+1}},\\
(\partial f)_a (x,b_0,\ldots,b_{p+1})=\sum_0^{p+1} (-1)^i f_a(x,b_0,\ldots,\hat b_i,\ldots,b_{p+1}),
\end{gather}
$(x,b_{0},\ldots,b_{p+1})\in\pi_p^{-1} U_a$.
The terms in (4.3) are all in different Banach spaces
$$
(\pi_{p+1}^* L)_{(x,b_0,\ldots,b_{p+1}) }\text{ and }(\pi_p^* L)_{(x,b_0,\ldots,\hat b_i,\ldots,b_{p+1})},
$$
and the equality of the two sides is understood after the canonical identification of these fibers with $L_x$.
Clearly, $\partial\delta=\delta\partial$, and $J$ is a subcomplex.

If $K^{pq}$ are endowed with the topology described in Section 2, and $J^{pq}$ with the topology induced by its identification with $\Gamma(S^{[p+1]},\pi_p^* L)$, then the embeddings $J^{pq}\to K^{pq}$ are continuous, as are the differentials $\delta,\partial$ on $K$ and $J$.
We write $\delta_J,\partial_J$ for the differentials restricted to the topological double complex $J$.

We can augment $K$ and $J$ by the columns
$$
K^{-1,\bullet}=C^\bullet (\fU,L),\qquad J^{-1,\bullet}=C^\bullet_{\hol} (\fU,L),
$$
or by the rows $K^{\bullet,-1}=J^{\bullet,-1}$, still defined by (4.1), which $\partial^{-1,\bullet}$ or $\delta^{\bullet,-1}$ map bijectively onto
$\Ker\partial^{0,\bullet}$, $\Ker\partial_J^{0,\bullet}$, resp.~$\Ker\delta^{\bullet,0}=\Ker\delta_J^{\bullet,0}$.

\begin{prop}For $p,q\geq 0$ the augmented complexes $K^{\bullet q},J^{\bullet q}$ and $J^{p\bullet}$ are exact.
Even better, for $n\geq -1$
\begin{gather*}
\partial^{nq} \colon  K^{nq}\to\Ker\partial^{n+1,q},\qquad \partial_J^{nq}\colon  J^{nq}\to\Ker\partial_J^{n+1,q},\text{ and}\\
\delta_J^{pn} \colon  J^{pn}\to \Ker\delta_J^{p,n+1}
\end{gather*}
have continuous linear right inverses, hence they are open maps.
\end{prop}

\begin{proof}If $f=(f_a)\in K^{n+1,q}$, $g=(g_a)\in J^{p,n+1}$, define $\lambda(f)\in K^{nq}$, $\mu(g)\in J^{pn}$ by
\begin{eqnarray*}
(\lambda f)_{a_0\ldots a_q} (x,b_0,\ldots,b_n)=\sum^q_{i=0}\ f_a (x,a_i,b_0,\ldots,b_n)/(q+1),\\
(\mu g)_{a_0\ldots a_n} (x,b_0,\ldots,b_p)=\sum^p_{i=0} g_{b_i a_0\ldots a_n} (x,b_0,\ldots,b_p)/(p+1).
\end{eqnarray*}
Then $\lambda|\Ker\partial$, $\lambda|\Ker\partial_J$, and $\mu|\Ker\delta_J$ are the required right inverses.
\end{proof}

This proof has little to do with the assumption that $S$ is Stein.
By contrast, the assumption is crucial in the next claim, which is [L, Theorem 4.6].

\begin{prop}For $p\geq 0$ the augmented complex $K^{p\bullet}$ is exact.
Even better, $\delta^{pn}\colon K^{pn}\to \Ker \delta^{p,n+1}$ is open for $n\geq -1$.
\end{prop}

\medskip
\noindent
{\begin{proof} [Proof of Theorem 3.1]
The total complexes $K^\bullet\subset J^\bullet$ of $K,J$ are given by
$$
K^n=\bigoplus^n_{p=0} K^{p,n-p},\ J^n=\bigoplus^n_{p=0} J^{p,n-p},
$$
on which the differentials are $d=\partial'+\delta,d_J=d|J^\bullet$, where $\partial^{'pq}=(-1)^p\partial^{pq}$.
In light of Propositions 4.1 and 4.2, the embeddings
\begin{align}
K^{-1,\bullet}=C^\bullet (\fU,L)&\overset{\partial^{-1,\bullet}}{\longrightarrow} K^\bullet, & J&^{-1,\bullet}=C^\bullet_{\hol} (\fU,L) \overset{\partial_J^{-1,\bullet}}{\longrightarrow} J^\bullet,\\
K^{\bullet,-1}\overset{\delta^{\bullet,-1}}{\longrightarrow} & K^\bullet,  & &J^{\bullet,-1}=K^{\bullet,-1} \overset{\delta_J^{\bullet,-1}}{\longrightarrow} J^\bullet
\end{align}
are quasi--isomorphisms of complexes of vector spaces, see e.g.~[S, Proposition 1, page 220].
Since the first map in (4.5) is the composition of the second with the inclusion $J^\bullet\hookrightarrow K^\bullet$, it follows that the latter is quasi--isomorphic.
As the diagram
\begin{eqnarray}
\begin{CD}
C^\bullet (\fU,L) & \overset{\partial^{-1,\bullet}}{\longrightarrow} & K^\bullet\\
\cup&&\cup\\
C^\bullet_{\hol} (\fU,L) & \overset{\partial_J^{-1,\bullet}}{\longrightarrow} & J^\bullet
\end{CD}
\end{eqnarray}
commutes, we deduce that the embedding $C^\bullet_{\hol}(\fU,L)\to C^\bullet (\fU,L)$ is a quasi--isomorphism of complexes of vector spaces.

The induced vector space isomorphism on cohomology is clearly continuous.
To show it is a topological isomorphism we need to verify it is open.
This follows by the same analysis as above, except that instead of [S, Proposition 1, page 220] we use [L, Proposition 2.4], in conjunction with the openness parts of Propositions 4.1, 4.2, to conclude (4.4) and (4.5) induce open maps in cohomology.
Hence, passing to cohomology in (4.6) we obtain a diagram in which three maps are topological isomorphisms, and therefore $H^n_{\hol} (\fU,L)\to H^n (\fU,L)$ must also be.

This finishes the proof of Theorem 3.1 when $H_{\hol}^n(\fU,L)$ is endowed with the finer topology.
But since the coarser topology is sandwiched between the finer one and the topology inherited from $H^n(\fU,L)$, the result follows for the coarser topology as well.

Finally, it is routine to check that the same argument works for cohomology groups that are defined in terms of alternating cochains.
\end{proof}

\section{Application:\ holomorphic group actions}

Suppose on a holomorphic Banach bundle $L\to X$ a complex reductive group $G$ acts holomorphically.
In this section we show how to decompose the cohomology groups $H^q(X,L)$ into isotypical subspaces, under certain assumptions on $X$ and the action.
The assumption will imply that $G$ acts on the cohomology groups, but the general theory of locally convex representations of reductive groups does not apply, because $H^q(X,L)$ is not guaranteed to be complete or Hausdorff.---This is not an issue when $X$ is compact and $L$ is of finite rank; but even line bundles over finite dimensional noncompact $X$ and Hilbert bundles over compact $X$ can exhibit non--Hausdorff cohomology groups.---
Instead, we shall work with a $G$--invariant Stein cover $\fU=\{U_a\colon a\in A\}$, and decompose the cochain groups of this cover.
The advantage is that the cochain groups are at least Hausdorff; however, on $C^q(\fU,L)$ the action is not holomorphic.
At this point enter the holomorphic cochains:\ on $C_{\hol}^q(\fU,L)$ the action is (often) holomorphic, and its isotypical decomposition will descend to a decomposition of cohomology groups because of Theorem 3.1.

We start with a Banach bundle $L\to X$ over a Stein manifold, and investigate the induced action on the locally convex space $\Gamma(X,L)$.
It will be convenient to consider actions not just of groups $G$ but of arbitrary sets, perhaps endowed with a topology or a manifold structure.

\begin{defn}An action of a set $G$ on $L$ means a collection of holomorphic, resp.~biholomorphic maps $\alpha_g\colon L\to L$, $\xi_g\colon X\to X$, $g\in G$, such that $\alpha_g
(L_{\xi_g x})\subset L_x$ for $x\in X$.
If $G$ is a topological space or a complex manifold, we say the action is continuous, resp.~holomorphic, if the maps
$$
\alpha\colon G\times L\ni (g,l)\mapsto\alpha_g (l)\in L,\qquad \xi\colon G\times X\ni (g,x)\mapsto \xi_g(x)\in X
$$
are continuous, resp.~holomorphic.
A continuous action is locally uniformly continuous if each $(g_0,x_0)\in G\times X$ has a neighborhood $G_0\times U$ such that  $\xi_g\to\xi_{\overline g}$ uniformly on $U$ and $\alpha_g\to\alpha_{\overline g}$ uniformly on a neighborhood of the zero section in $L_U$, as $g\to\overline g\in G_0$.
\end{defn}

The notion of uniform convergence for maps with values in a manifold is understood after the target manifold is locally identified with a Banach space, so that for small enough $G_0$ and $U$ the maps in question can be thought to take values in a Banach space.

Thus a left action of a group is an action that respects the group structure in the sense that $\alpha_e=\id_L$ and $\alpha_{gh}=\alpha_g\alpha_h$ (which implies $\xi_{gh}=\xi_h\xi_g$, so $\xi$ is a right action).
Clearly, in general $\alpha_g$ determines $\xi_g$ uniquely, and one can talk of $\alpha$ as the action.

An action $(\alpha,\xi)$ on $L$ determines an action $\beta$ on $\Gamma(X,L)$ by
\begin{equation}
(\beta_g f)(x)=\alpha_g (f(\xi_g x)),\qquad g\in G,f\in\Gamma(X,L),x\in X.
\end{equation}

\setcounter{lemma}{+1}
\begin{lemma}Suppose $X$ is second countable and $G$ is a locally compact topological space, resp.~a finite dimensional complex manifold.
If a $G$--action $\alpha$ on $L$ is locally uniformly continuous, resp.~holomorphic, then the induced action on $\Gamma(X,L)$,
\begin{equation}
\beta\colon G\times\Gamma (X,L)\ni (g,f)\mapsto (\beta_g f)\in\Gamma (X,L)
\end{equation}
is continuous, resp.~holomophic.
\end{lemma}

By Theorem 2.3 and Lemma 2.2 we can assume that $L$ is a direct summand in a trivial bundle $T=X\times E\to X$, where $(E, \|\ \|_E)$ is a Banach space.
We denote by $\|\ \|\colon L\to [0,\infty)$ the restriction of $\|\ \|_E$ to the fibers $L_x\subset \{x\}\times E$.
According to Proposition 2.4 the topology on $\Gamma(X,L)$ is induced from the embedding $\Gamma(X,L)\subset\Gamma(X,T)$.
Concretely, this means that with weights $w\in W(X)$ if we define Banach spaces $(\Gamma(w,L),\ \|\ \|_w)$,
$$
\Gamma(w,L)=\{f\in\Gamma(X,L)\colon \|f\|_w=\sup_{x\in X} \|f(x)\|/w(x)<\infty\},
$$
then $\Gamma(X,L)=\underset{\longrightarrow}{\lim}\ \Gamma (w,L)$ as locally convex spaces.
To prove Lemma 5.2 we need

\begin{lemma}If $X$ is second countable, a compact space $G$ acts locally uniformly on $L$, and $\beta$ is the induced action on $\Gamma(X,L)$, then for every $w\in W(X)$ there is a $w'\in W(X)$ such that
\begin{equation}
G\times\Gamma(w,L)\ni (g,f)\mapsto\beta_g f\in\Gamma(w',L)
\end{equation}
is continuous with respect to the $\|\ \|_w, \|\ \|_{w'}$ topologies on $\Gamma(w,L)$, $\Gamma(w',L)$.
\end{lemma}

\begin{proof}The action $\alpha$ on $L$ can be extended to an action $\tilde\alpha$ on the trivial bundle $T=L\oplus L'$ by letting $\tilde\alpha_g$ send $L'_x$ to the zero vector in $L'_{\xi_g^{-1}x}$.
Hence it suffices to prove the lemma for $T$ instead of $L$, or to put it differently, we can assume $L=X\times E\to X$ is trivial.

First we claim there is a $w_1\in W(X)$ such that
\begin{equation}
\|\alpha (g,l)\|\leq w_1(x) \|l\|,\qquad g\in G,l\in L_x,x\in X.
\end{equation}
Indeed, by continuity of $\alpha$, for any $(g_0,x_0)\in G\times X$ there are a neighborhood $G_0\times U$ and an $\var > 0$ such that
$$
\|\alpha(g,l)\|\leq 1,\quad (g,l)\in G_0\times L_U,\ \|l\|\leq\var.
$$
Hence $\|\alpha(g,l)\|\leq \|l\|/\var$ for $(g,l)\in L_U$.
Since $G$ can be covered by finitely many such $G_0$, we obtain for every $x_0\in X$ a neighborhood $U\subset X$ and a positive number $c_U$ such that $\|\alpha(g,l)\|\leq c_U\|l\|$ for $l\in L_U$.
Denoting by $\fU$ a cover of $X$ by such neighborhoods, the weight $w_1(x)=\inf \{c_U\colon x\in U\in\fU\}$ will do.

Set $w_2(x)=\sup_{g\in G} w(\xi_g x) w_1 (\xi_g x)$, a locally bounded positive function.
(5.4) implies for $f\in\Gamma (w,L)$ 
\begin{equation}
\begin{gathered}
\begin{split}\|\beta_g  f(x) \|=&\|\alpha_g  (f(\xi_g x))\|\leq\\
&{\|f(\xi_g x)\|\over w(\xi_g x)} w(\xi_g x) w_1 (\xi_g x)\leq \|f\|_w w_2(x),\quad\text{ and}\end{split}\\
\|\beta_g f\|_{w'}\leq \|f\|_w,
\end{gathered}
\end{equation}
provided $w'\geq w_2$.
In particular, $\beta_g$ is continuous for fixed $g$.

Next we claim that every $x_0\in X$ has a neighborhood $U$ such that with $\overline g\in G$ and $f\in\Gamma(w,L)$
\begin{equation}
\|(\beta_g f)(x)-(\beta_{\overline g} f)(x)\|\to 0\text{ uniformly for }x\in U,\text{ as }g\to\overline g.
\end{equation}
At first we prove a weaker version when a $g_0\in G$ is also fixed, and (5.6) is claimed only for $\overline g$ in some neighborhood $G_0\subset G$ of $g_0$.

Since $L=X\times E\to X$ is trivial, we can write $f\in\Gamma(w,L)$ as $f(x)=(x,e(x))$, where $e\in{\cO}^E(X)$.
Similarly,
\begin{equation}
\alpha_g(l)=(\xi^{-1}_g x,a(g,l)),\qquad g\in G, l\in L_x,
\end{equation}
where $a\colon G\times L\to E$.
As $\alpha$ is locally uniformly continuous, we can choose neighborhoods $U_1\subset X$ of $x_0$, $V_1$ of $\xi_{g_0} x_0$, and $G_1\subset G$ of $g_0$ such that for $g\to\overline g\in G_1$
\begin{equation}\begin{gathered}
 \xi_g\to\xi_{\overline g}\ \text{ uniformly on }U_1,\\
 a(g,l)\to a(\overline g,l)\text{ uniformly for }l\in V_1\times B,
\end{gathered}\end{equation}
where $B$ is a neighborhood of $0\in E$.
By homogeneity, the same holds for any bounded set $B\subset E$.
We arrange that
\begin{equation*}
\sup_{V_1} w=s_1 < \infty,\text{ and}
\sup_{G\times V_1\times B'} \|a\|_E < \infty,
\end{equation*}
with some neighborhood $B'$ of $0\in E$.
As $a$ is linear on the fibers of $L$, the latter implies $\|a(g,l)\|_E\leq s_2 \|l\|$ for $g\in G$, $l\in L_{V_1}$, with $s_2<\infty$.
Since a neighborhood of $x_0\in X$ can be identified with a ball in some Banach space $F$, we can also arrange that $V_1$ is identified with the unit ball of $F$, the center corresponding to $\xi_{g_0}x_0$.
Let $V\subset V_1$ denote the concentric ball of radius 1/2, and choose $G_0\subset G_1$, $U\subset U_1$ so that $G_0 U\subset V$.
The operator norm of $\partial a/\partial l$ can then be estimated \`a la Cauchy:
\begin{equation}
\left\|{\partial a(g,l)\over\partial l}\right\|_{\text{op}}\leq 2s_2 (\|l\|+1),\ g\in G,\ l\in L_V.
\end{equation}

Further, if $f\in\Gamma (w,L)$ and $f(x)=(x,e(x))$, then
\begin{equation}
\sup_{V_1} \|e\|_E\leq s_1 \|f\|_w\text{ and }\sup_V \left\|{\partial e\over\partial x}\right\|_{\text{op}}\leq 2s_1\|f\|_w.
\end{equation}
With $g,\overline g\in G_0$, $x\in U$
\begin{multline*}
\|(\beta_g f)(x)-(\beta_{\overline g}f)(x)\| \leq\\
\|\alpha_g (f(\xi_g x))-\alpha_g (f(\xi_{\overline g} x))\|+\|\alpha_g (f(\xi_{\overline g} x))-\alpha_{\overline g}(f(\xi_{\overline g} x))\|,
\end{multline*}
and the first term on the right is
\begin{multline*}
\|a(g,\xi_g x, e(\xi_g x))-a(g, \xi_{\overline g} x, e(\xi_{\overline g} x))\|_E\\
\leq 2 s_2 (s_1 \|f\|_w+1) (2s_1 \|f\|_w+1) \|\xi_g x-\xi_{\overline g} x\|_F
\end{multline*}
by (5.9), (5.10).
Hence (5.8) implies this term tends to $0$ uniformly as $g\to\overline g$.
The second term, or $\|a(g,\xi_{\overline g}x, e(\xi_{\overline g} x))-a(\overline g,\xi_{\overline g} x,e(\xi_{\overline g} x))\|_E$, also tends to $0$ uniformly, by the second part of (5.8) and by (5.10).
Therefore indeed $\|\beta_g f-\beta_{\overline g} f\|\to 0$ uniformly on $U$ as $g\to\overline g\in G_0$.

Since $G$ can be covered by finitely many such $G_0$, the intersection of the corresponding $U$'s provides the neighborhood $U$ for which (5.6) holds for all $\overline g\in G$ and $f\in\Gamma (w,L)$.

Now we cover $X$ by a locally finite family of such open sets $U$.
The family must be countable, let us denote its elements $U_1,U_2,\ldots$, and define $w'\in W(X)$ by
$$
w'(x)=\max \{n w_2 (x)\colon x\in U_n\}.
$$
For $x\in U_n\cup U_{n+1}\cup\ldots$ (5.5) implies
$$
\| (\beta_g f) (x)\|/w'(x)\leq \|f\|_w/n,
$$
so that it follows from (5.6) with $U=U_1,\ldots,U_{n-1}$ that $\|\beta_g f-\beta_{\overline g} f\|_{w'}\to 0$.
Putting this and (5.5) together
$$
\|\beta_g f-\beta_{\overline g}\overline f\|_{w'}\leq \|\beta_g (f-\overline f)\|_{w'} + \|\beta_g \overline f-\beta_{\overline g}\overline f\|_{w'}\to 0
$$
as $g\to\overline g$ and $f,\overline f\in\Gamma (w,L)$, $f\to\overline f$; the induced action is indeed continuous.
\end{proof}
\begin{proof}[Proof of Lemma 5.2]
(a)\ Since continuity is a local property, we might as well assume $G$ is compact, and then continuity follows from Lemma 5.3.

(b)\ Holomorphy is also a local property, so now we can assume $G$ is an open subset in some $\bC^m$.
Cauchy estimates imply that a holomorphic action is locally uniformly continuous, hence $\beta$ is continuous.
To prove it is holomorphic as well, we need to compute its directional derivatives $d\beta(\gamma,\varphi)\in T\Gamma(X,L)$ for $\gamma\in TG$, $\varphi\in T\Gamma (X,L)$, and prove $d\beta$ is continuous.
The action $(\alpha,\xi)$ of $G$ on $L$ defines a holomorphic action $(\alpha',\xi')=(d\alpha,d\xi)$ of $TG$ on the bundle $TL\to TX$ and an action $\beta'$ on $\Gamma(TX,TL)$,
\begin{equation}
\beta'(g',f')(x')=\alpha'\big(g',f'(\xi'(g',x'))\big),
\end{equation}
where $g'\in TG$, $f'\in\Gamma (TX,TL)$, and $x'\in TX$, see (5.1), (5.2).
By part (a), $\beta'$ is continuous.

Let $V=\bigcup_{x\in X} TL_x\subset TL$ denote the subbundle of vertical vectors.
If $f\in\Gamma(X,L)$, there are natural topological isomorphisms among $\Gamma(X,L)$, $\Gamma(X,f^* V)$, and $T_f\Gamma (X,L)$.
Indeed, if $\varphi\in\Gamma (X,L)$, for every $x\in X$ the tangent vector to the curve $t\mapsto f(x)+t\varphi(x)$, at $t=0$, is in $f^* V_x$, and so $\varphi$ determines a holomorphic section of $f^* V$.
Similarly, the tangent vector to the curve $t\mapsto f+t\varphi$ is in $T_f \Gamma (X,L)$.
In what follows, we shall not distinguish between $T_f\Gamma (X,L)$ and $T(X,f^* V)$.
Similarly, $T\Gamma (X,L)$ will be identified with $\Gamma (X,V)\subset\Gamma (X, TL)$.

For fixed $x$ the map $G\times\Gamma (X,L)\ni (g,f)\mapsto (\beta_g f)(x)\in L$ is differentiable and by the chain rule its derivative in the direction $(\gamma,\varphi)\in T_g F\times T_f\Gamma (X,L)$ is
\begin{equation}
d\alpha\big(\gamma,df(d\xi(\gamma,0_x))+\varphi(\xi(g,x))\big),
\end{equation}
where $0_x\in T_x X$ stands for the zero vector.
In fact, the directional difference quotients converge locally uniformly in $x\in X$, whence the difference quotients of $\beta$ converge in the topology of $\Gamma(X,L)$.
We conclude that $d\beta(\gamma,\varphi)\in T\Gamma (X,L)\approx T(X,V)$ exists, and $d\beta(\gamma,\varphi)(x)$ is given by (5.12).
Now, if with each $\varphi\in T_f\Gamma (X,L)\approx \Gamma (X,f^* V)$ we associate $\varphi'\in\Gamma (TX,TL)$,
$$
\varphi' (x')=df (x')+\varphi(x),\qquad x'\in T_x X,
$$
then by (5.11), (5.12) we see that $d\beta(\gamma,\varphi)$ is the restriction of $\beta'(\gamma,\varphi')\in\Gamma (TX,TL)$ to the zero section $X\subset TX$.
Since both $\beta'$ and the map $\varphi\mapsto\varphi'$ are continuous, it follows that so is $d\beta$, and $\beta$ is indeed holomorphic.
\end{proof}

We are interested in actions of finite dimensional complex Lie groups $G$ on Banach bundles.
We assume that $G$ is reductive in the sense that there is a compact real subgroup $G_{\bR}\subset G$ which is maximally real and intersects each component of $G$.
Recall that a real submanifold $N\subset M$ of a complex manifold is maximally real if $T_x N\oplus i T_x N=T_x M$ for all $x\in N$.
When $\dim M<\infty$, this amounts to requiring that locally the pair $(M,N)$ is biholomorphic to $(\bC^n,\bR^n)$.
Let $\text{Irr } = \text{ Irr}_{G_{\bR}}$ denote the set of irreducible characters of $G_{\bR}$.
Any holomorphic irreducible representation $G\to\text{ GL}(k,\bC)$ restricts to an irreducible representation of $G_{\bR}$, and the representation of $G$ can be recovered from its restriction.
Hence isomorphism classes of irreducible representations of $G$ can be labeled by (certain) characters $\chi\in\text{ Irr}$.
Given an arbitrary representation of $G$ on a complex vector space $V$, the $\chi$--isotypical subspace $V_\chi\subset V$ is the linear span of the subrepresentations labeled by $\chi$.

\setcounter{thm}{+3}
\begin{thm}Suppose $X$ is a second countable Stein manifold and a complex reductive group $G\supset G_{\bR}$ acts holomorphically on a holomorphic Banach bundle $L\to X$.
Then the isotypical subspaces $\Gamma_\chi (X,L)\subset\Gamma (X,L)$ of the induced representation $\beta$ are closed for $\chi\in\text{ Irr}$, 
there are surjective, continuous, $G$--equivariant operators $P_\chi\colon\Gamma (X,L)\to\Gamma_\chi (X,L)$,
$$
P_\chi P_\psi=\begin{cases} P_\chi&\text{if $\chi=\psi$}\\
0 &\text{if $\chi\neq\psi$,}\end{cases}
$$
and $\bigoplus_\chi\Gamma_\chi (X,L)\subset\Gamma(X,L)$ is dense.
\end{thm}

\begin{proof}Recall that given two holomorphic $G$--representations on Hausdorff locally convex spaces $V$ and $Z$, any $G_{\bR}$ equivariant continuous operator $P\colon V\to Z$ is automatically $G$--equivariant.
Indeed, if $v\in V$ then the map $\varphi\colon G\ni g\mapsto g^{-1} Pgv\in Z$ is holomorphic, and $\varphi(g)=Pv$ if $g\in G_{\bR}$.
As $G_{\bR}$ is maximally real and intersects each component of $G$, any holomorphic function on $G$ is uniquely determined by its restriction to $G_{\bR}$, whence $\varphi\equiv Pv$, i.e.~$P$ is $G$--equivariant.
Similarly, if $Y\subset V$ is a finite dimensional (hence closed) $G_{\bR}$--invariant subspace, then it is also $G$--invariant; for if $\pi$ denotes projection $V\to V/Y$, then for $v\in Y$ the holomorphic function $g\mapsto\pi (gv)$ vanishes for $g\in G_{\bR}$, hence vanishes for all $g\in G$.
This shows that in the theorem we can forget about representations of $G$ and deal with representations of $G_{\bR}$ only.

In light of this, the theorem follows from general representation theory,
see [BD, III.5].
The operators $P_\chi$ are defined by a Haar integral:
\begin{equation}
P_\chi f=\dim\chi\int_{G_{\bR}}\chi(g^{-1})(\beta_g f) dg\in\Gamma(X,L),\qquad f\in\Gamma(X,L),
\end{equation}
which makes sense since $\Gamma(X,L)$ is sequentially complete, see Proposition 2.5.
From (5.13) one proves that $\Gamma_\chi (X,L)=P_\chi\Gamma (X,L)$ and the other claims in our Theorem as in [BD, III.5].
Theorem 5.10 there is formulated for Hilbert representations only, but the relevant parts of its proof give what we need in our Theorem 5.4.
\end{proof}

This has a consequence for the isotypical decomposition of cohomology groups.
Suppose $A,X$ are second countable complex manifolds, $\fU=\{U_a\colon a\in A\}$ an open cover of $X$, and
$$
S=\{(a,x)\in A\times X\colon x\in U_a\}\subset A\times X
$$
is a Stein open subset (necessarily second countable).
Let $L\to X$ be a holomorphic Banach bundle.
Assume a complex reductive Lie group $G\supset G_{\bR}$ acts holomorphically on the right on $A$ and on the left on $L$.
We write $(g,a)\mapsto ag$ for the action on $A$, while retain the notation $\alpha,\xi$ for the actions on $L$ and $X$.
Finally assume that the actions on $A$ and $X$ are compatible:\ $\xi_g U_a=U_{ag}$.
Then, in the notation of Section 3, the diagonal action on $A^{p+1}\times X$ restricts to a holomorphic action $\xi^p$ on $S^{[p+1]}$, and the action on $L$ can be lifted to a holomorphic group action $\alpha^p$ on $\pi_p^* L$, $p=0,1,\ldots$.
The upshot is that there is an induced holomorphic representation $\beta^p$ on $\Gamma(S^{[p+1]},\pi_p^* L)\approx C^p_{\hol} (\fU,L)$.
Clearly, \v Cech coboundary is $G$--equivariant.

\setcounter{thm}{+4}
\begin{thm}With notation and assumptions as above, the isotypical subspaces $V_\chi^p\subset C^p_{\hol} (\fU,L)$ are closed for $\chi\in\text{ Irr}$, there are surjective, continuous, $G$--equivariant operators $P_\chi^p\colon C_{\hol}^p (\fU,L)\to V_\chi^p$,
$$
P_\chi^p P_\psi^p=\begin{cases} P_\chi^p &\text{if $\chi=\psi$}\\
0 &\text{if $\chi\neq\psi$, }\end{cases}
$$
and $\bigoplus_\chi V_\chi^p\subset C^p (X,L)$ is dense.
The same holds if the space $C^p_{\hol} (\fU,L)$ of cochains is replaced by the subspace of cocyles or by the cohomology groups $H^p(X,L)$.
\end{thm}

\begin{proof}The first part is an immediate consequence of Theorem 5.4.
The second part follows because the projections
$$
P_\chi^p=\dim\chi\int_{G_{\bR}}\chi(g^{-1})\beta_g^p dg,
$$
cf.~(5.14), respect \v Cech coboundary:\ $\delta^p P_\chi^p=P_\chi^{p+1}\delta^p$, and because $H_{\hol}^p (\fU,L)\approx H^p (X,L)$, cf.~Theorem 3.1 and [L, Theorem 4.5].
\end{proof}

\end{document}